\documentclass{article}
\usepackage{amsmath,amsthm,amssymb,amscd}

\newtheorem{Thm}{Theorem}[section]
\newtheorem{Lem}[Thm]{Lemma}
\newtheorem{Prop}[Thm]{Proposition}
\newtheorem{Def}[Thm] {Definition}

\newtheorem{Cor}[Thm]{Corollary}
\newtheorem{Rem} [Thm]{Remark}

\theoremstyle{claim}

\DeclareMathOperator{\Diff}{Diff}

\def\u{{\cal U}}

\def\v{{\cal V}}

\def\e{{\cal E}}
\def\r{{\cal R}}

\begin{document}
\begin{center}
{\Large \bf Some results on perturbations to Lyapunov exponents}\\
\smallskip
CHAO LIANG$\dagger$ and
WENXIANG SUN$\ddagger$ and
JIAGANG YANG$\ast$\\
\smallskip
$\dagger$Applied Mathematical Department, The Central University of Finance
and Economics, Beijing 100081, China\\
(chaol@cufe.edu.cn)\\

$\ddagger$LMAM, School of Mathematical Sciences, Peking University, Beijing 100871, China\\
(sunwx@math.pku.edu.cn)\\
$\ast$Departamento de Geomoetria, Instituto de Matem$\acute{a}$tica,
Universidade Federal Fluminense, 24020-140, Niter$\ddot{o}$i,
Brazil\\
( yangjg@impa.br)

\smallskip

\end{center}

\footnotetext {$\dagger$Liang is supported by NNSFC(\# 10901167 )}
\footnotetext {$\ddagger$Sun is supported by NNSFC(\# 10231020) and
National Basic Research Program of China(973 Program)(\#
2006CB805903) } 
\footnotetext {AMS Review: 37C40; 37D25; 37H15; 37A35}

\bigskip

\begin{abstract}
In this paper, we study two properties of the Lyapunov exponents under small perturbations: one is when we can remove zero
 Lyapunov exponents and the other is when we can distinguish all the Lyapunov exponents. The first result shows that we
 can perturb all the zero integrated Lyapunov exponents $\int_M \lambda_j(x)d\omega(x)$ into nonzero ones, for any  partially hyperbolic
 diffeomorphism. The second part contains an example which shows the local genericity of  diffeomorphisms with non-simple spectrum and three results: one discusses the relation between simple-spectrum property and the existence of complex eigenvalues; the other two describe the difference on the spectrum between the diffeomorphisms far from homoclinic tangencies and those in the interior of the complement. Moreover, among the conservative diffeomorphisms far from tangencies, we prove that ergodic ones form a residual subset.
 \end{abstract}

\section{Introduction}
\bigskip

It was shown in the 1960s that uniformly hyperbolic systems are not dense in the space of dynamical systems. This brought about the naissance of the notions of partial and nonuniform hyperbolicity. Using the concept of Lyapunov exponents, Pesin theory of nonuniformly hyperbolic systems(characterized by all the Lyapunov exponents are non-null for some invariant measure) gives us a rich information about geometric properties of the system. In particular, the points with all Lyapunov exponents non-zero have well-defined unstable and stable invariant manifolds. These tools are the base of most of the results on dynamical systems nowadays. Thus it is of utmost importance to detect when the zero exponents can be removed by perturbations.

One central result in this direction for discrete system is the
Shub-Wilkinson example\cite{SW}. They build a conservative
perturbation to a skew product of an Anosov diffeomorphism of the
torus $\mathcal{T}^2$ by rotations and creates positive exponents in
the center direction for Lebesgue almost every point.
Baraviera-Bonatti present a local version of Shub-Wilkinson's
argument, allowing one to perturb the sum of all the center
integrated Lyapunov exponents of any conservative partially
hyperbolic systems.

Highly inspired by their results we prove in the present  paper
that 
one can perturb every integrated Lyapunov exponent $\int_M
\lambda_j(x)d\omega(x)$ of any conservative partially hyperbolic
systems to nonzero ones, not only the {\bf sum} of them.  This is a
generalization to Baraviera-Bonatti\cite{BB}.

Let $M$ be a $d-$dimensional compact Riemannian manifold without
boundary ($d\geq2$) and $\omega$ be a smooth volume form. Denote by
$\Diff^1_\omega(M)$ the set of volume preserving
$C^1-$diffeomorphisms on $M$. Take $f\in \Diff^1_\omega(M).$ By
Oseledec Theorem, there exists a $Df$-invariant decomposition
$$T_xM=\bigoplus\limits_{i=1}^{s(f,x)} E^i(x)\quad\omega-a.e.x\in M\eqno(1.1)$$ such that
 $$\lim\limits_{m\to
\pm\infty}\frac{1}{m}log\|d_xf^m\nu\|=\lambda_i(f,x)$$ converges
uniformly on $\{\nu\in E^i(x):\|\nu\|=1\}$, where $1\leq s(f,x)\leq
d$. This number $\lambda_i(f,x)$ is called the Lyapunov exponent
associated with $E^i$. Lyapunov exponents describe the asymptotic
evolution of tangent map: positive or negative exponents correspond
to exponential growth or decay of the norm, respectively, whereas
vanishing exponents mean lack of exponential behavior.  None of
these values depends on the choice of a Riemannian metric.

Throughout this paper, let $\lambda_{1}(f,x)\geq\lambda_2(f,x)\geq
\cdots\geq\lambda_d(f,x)$ be the Lyapunov exponents in nonincreasing
order and each repeated with multiplicity $dimE^i(x)$. We define
{\bf the $j-th-$integrated exponent} by $\int_M
\lambda_j(f,x)d\omega(x)$, for any $1\leq j\leq d$.


\bigskip

A $Df-$invariant splitting $TM=E^1\oplus\cdots\oplus E^k$ is called
a {\em dominated splitting} if each $E_i$ is a continuous
$Df-$invariant subbundle of $TM$ and if there is some integer $n>0$
such that, for any $x\in M$, any $i<j$ and any non-zero vectors
$u\in E^i(x)$ and $\nu\in E^j(x)$, one has
$$\frac{\|Df^n(u)\|}{\|u\|}<\frac{1}{2}\frac{\|Df^n(\nu)\|}{\|\nu\|}.$$

Let $\mathcal{PH}^1_\omega(M)$ denote the subset of
$\Diff^1_\omega(M)$ consisting of all the partially hyperbolic
diffeomorphisms.
In this paper, partially hyperbolic means the following. There is a nontrivial splitting of the tangent bundle, $TM=E^u\oplus E^c\oplus E^s$, that is invariant under the derivative map $Df$. Further, there is a Riemannian metric for which we can choose continuous positive functions $\nu$, $\tilde\nu$, $\gamma$ and $\tilde\gamma$ with \begin{eqnarray*}
\nu,\tilde\nu<1\mbox{ and } \nu<\gamma<{\tilde\gamma}^{-1}<{\tilde\nu}^{-1}
\end{eqnarray*}
such that, for any unit vector $\upsilon\in T_pM$,

\setcounter{equation}{1}
\begin{eqnarray}
&\|Df\upsilon\|&<\nu(p), \mbox{ if } \upsilon\in E^s(p),\\
\gamma(p)<&\|Df\upsilon\|&<{\tilde\gamma(p)}^{-1},\mbox{ if }\upsilon\in E^c(p),\\
{\tilde\nu (p)}^{-1}<&\|Df\upsilon\|&, \mbox{ if }\upsilon\in E^u(p).
\end{eqnarray}

\begin{Thm}\label{Thm:nonzeroexp}\,\, Let
$f\in\mathcal{PH}^1_\omega(M)$.
Then for any
neighborhood $\u\subseteq\mathcal{PH}^1_\omega(M)$ of $f$, there is a
diffeomophism $g\in\u$, such that every integrated Lyapunov exponent
is different from zero, i.e.,
$$\int_M \lambda_j(g,x)d\omega(x)\neq0,\quad\forall 1\leq j\leq d.$$
\end{Thm}

\bigskip

Much work on perturbation of Lyapunov exponents concerns two basic
topics. One is when we can remove zero Lyapunov exponents as we
discussed above. The other is when we can distinguish all the
Lyapunov exponents. That is, the spectrum is simple. About the
latter topic, all work are concentrated for cocycles. In the case of
independent and identically distributed random matrices, Arnold and
Cong\cite{AC1} showed that the cocycles with simple Lyapunov
spectrum form a residual set. In the space of all linear cocycles
equipped with the uniform topology, Knill\cite{K} (for the
2-dimensional case) and Arnold and Cong\cite{AC2} (for the general
$d-$dimensional case) proved that they form a dense set. For
$SL(d,\,\mathbb{R})-$cocycles, Bonatti and Viana\cite{BoV} also
showed the  density property under some conditions. It is natural to
ask  whether the   diffeomorphisms with simple spectrum form a dense
set in $\Diff^1_\omega(M)$. In Section 3, we construct an
example related to this problem and show the local genericity of diffeomorphisms with non-simple spectrum. Using the techniques developed in \cite{BDP} and \cite{LL}, we prove a dichotomy between simple-spectrum property and the existence of complex eigenvalues.
\smallskip

\begin{Thm}\label{Thm:dichotomy}
There is a residual subset $\mathcal{R}\subset \mathcal{PH}^1_\omega(M)$ of diffeomorphisms $f$ such that
\begin{enumerate}
\item[$\bullet$]either the finest dominated splitting of $E^c$ is of the form $$E^c(x)=\oplus^c_{i=1}E^{ci}(x)\mbox{ with }dim E^{ci}(x)=1,\mbox{ for all }i\in\{1,2,...,c\},\eqno(1.5)$$ $\omega-a.e.x\in M,$ where $c=dim E^c$.
    Hence, the multiplicity of every center Lyapunov exponent of $f$ is 1;

\item[$\bullet$]or for any $\varepsilon>0$, there is an $\varepsilon-$perturbation $g\in \mathcal{PH}^1_\omega(M)$ of $f$ and a periodic point $q$ of $g$ such that $Dg^{p(q)}_q$ has a complex center eigenvalue
, where $p(q)$ is the period of $q$.
\end{enumerate}
\end{Thm}

\bigskip

In the previous theorem, the first case is that each invariant subbundle in the finest dominated splitting is 1-dimensional. For the center submanifold, this property is much stronger than that of simple spectrum. Results in \cite{W}\cite{LLS}\cite{LL2} indicates that, roughly speaking, dominated splitting and homoclinic tangency are mutually exclusive concepts. We prove two results which suggest that diffeomorphisms far from homoclinic tangencies and in the interior of the complement load bearing entirely different properties. Moreover, applying the novel approach developed by F.R.Hertz, M.A.R.Hertz, A.Tahzibi, and R.Ures\cite{HHTU, HHTU2} for Pugh-Shub's stable ergodicity conjecture, we point out the genericity of the ergodic diffeomorphisms in the conservative ones which are far from tangencies.
\smallskip

We say a diffeomorphism $f$ has $C^1$ homoclinic tangency if $f\in \Diff^1(M)$ has hyperbolic periodic point $p$ at which the stable and unstable invariant manifolds $W^s(p)$ and $W^u(p)$ intersect non-transversely. Denote by $HT$ the set containing all the diffeomorphisms with homoclinic tangencies. We call a diffeomorphism $f$ is far away
from tangency if $f\in\Diff^1(M) \setminus \overline{HT}$.

In the context of the following two results, we suppose $d=dim M\geq 3$.

\smallskip

\begin{Thm}\label{t.ergodicfarfromtangency}
 There is a $C^1$ residual subset $\mathcal{R}$ of the volume preserving diffeomorphsims far from
tangencies, such that any $f\in \mathcal{R}$ admits a finest dominated splitting
$E^s\oplus E^{c1}\oplus \cdots \oplus E^{cc}\oplus E^u$, where $dim(E^{ci})=1$ and $c=dimE^c$.
Moreover, $f$ is ergodic.
\end{Thm}
\bigskip

\begin{Cor}There is a $C^1$ residual subset $\mathcal{R}$ of the volume preserving diffeomorphsims far from
tangencies, such that each center Lyapunov exponent of $f\in \mathcal{R}$ has multiplicity 1.
\end{Cor}

\bigskip

Recall a partially hyperbolic diffeomorphism is center bunched if the
functions $\nu$, $\tilde\nu$, $\gamma$ and $\tilde\gamma$ in (1.2)-(1.4) can be chosen so that:
$$\nu<\gamma\tilde\gamma\quad\mbox{    and    }\quad \tilde\nu<\gamma\tilde\gamma.$$
Let $\mathcal{CPH}_\omega^1(M)$ denote the set of all center bunching, conservative, partially hyperbolic diffeomorphisms. The statement of the third result about spectrum is the following:

\begin{Thm}\label{Thm:interiorofHT}
There is a $C^1$ residual subset of diffeomorphism $\mathcal{R}$ in
$\mathcal{CPH}_\omega^1(M)\cap int(\overline{HT})$, such that for
any $f\in \mathcal{R}$, it is ergodic, and it has two center
exponents equal.
\end{Thm}

\bigskip

\section{Proof of Theorem \ref{Thm:nonzeroexp}}
\bigskip

Given a vector space $V$ and a positive integer $p$, let $\wedge^p(V)$ be the $p-th$ exterior power of $V$. This is a vector space of dimension $C^p_d$, whose elements are called $p-vectors$. It is generated by the $p-$vectors of the form $\nu_1\wedge\cdots\wedge\nu_p$ with $\nu_j\in V$, called the {\em decomposable $p-$vectors}. A linear map $L:V\to W$ induces a linear map $\wedge^p(L):\wedge^p(V)\to\wedge^p(W)$ such that $$\wedge^p(L)(\nu_1\wedge\cdots\wedge\nu_p)=L(\nu_1)\wedge\cdots \wedge L(\nu_p).$$
If $V$ has an inner product, then we endow $\wedge^p(V)$ with the inner product such that $\|\nu_1\wedge\cdots\wedge\nu_p\|$ equals the $p-$dimensional volume of the parallelepiped spanned by $\nu_1,\cdots,\nu_p$.

More generally, there is a vector bundle $\wedge^p(\e)$, with fibers $\wedge^p(\e_x)$, associated to $\e$, and there is a vector bundle automorphism $\wedge^p(F)$, associated to $F$. If the vector bundle $\e$ is endowed with a continuous inner product, then $\wedge^p(\e)$ also is. The Oseledec data of $\wedge^p(F)$ can be obtained from that of $F$, as shown by the proposition below.

\smallskip
\begin{Prop}[Theorem 5.3.1 in \cite{AR}]\label{Prop:OseleExterior}
The Lyapunov exponents ( with multiplicity )
$\lambda_i^{\wedge^p}(x),\,\,1\leq i\leq C^p_d$, of the automorphism
$\wedge^p(F)$ at a point $x$ are the numbers
$$\lambda_{i_1}(x)+\cdots+\lambda_{i_p},\quad\mbox{ where } 1\leq
i_1<\cdots< i_p\leq d.$$ Let $\{e_1(x),\cdots,e_d(x)\}$ be a basis
of $\e_x$ such that $$e_i(x)\in E_x^\ell\quad \mbox{ for }
dimE^1_x+\cdots+dimE_x^{\ell-1}<i\leq
dimE^1_x+\cdots+dimE_x^{\ell}.$$ Then the Oseledec space
$E_x^{j,\wedge^p}$ of $\wedge^p(F)$ corresponding to the Lyapunov
exponent $\hat{\lambda}_j(x)$ is the sub-space of $\wedge^p(\e_x)$
generated by the $p-th$ vectors $$e_{i_1}\wedge\cdots\wedge e_{i_p},
\quad\mbox{ with }1\leq i_1<\cdots<i_p\leq d\quad\mbox{ and }
\lambda_{i_{1}}+\cdots+\lambda_{i_p}=\hat{\lambda}_j(x).$$
\end{Prop}
\bigskip

A dominated splitting $E^1\oplus\cdots\oplus E^n$ is called the
finest dominated splitting if there is no dominated splitting
defined over all $M$ in each invariant bundle $E^i$, $1\leq i\leq
n$. The continuation of the finest dominated splitting may not be
the finest dominated splitting of the perturbation diffeomorphism.
However, we can perturb any partially hyperbolic diffeomorphism to
obtain a diffeomorphism with robust finest dominated splitting.

\begin{Lem}\label{Lem:finestdom}
The diffeomorphisms with robust finest dominated splitting are $C^1$
dense among  the $C^r$, $r\geq1$, partially hyperbolic
diffeomorphisms of $M$.
\end{Lem}
\smallskip

{\bf proof}\,\, For any neighborhood $\u$ of $f$ in
$\mathcal{PH}^1_\omega(M)$, we denote $\u_f\subset\u$ be the
neighborhood of $f$ in $\mathcal{PH}^1_\omega(M)$ such that the
continuation of the dominated splitting of $f$ is a dominated
splitting of the perturbation diffeomorphism. Assume that there is a
small perturbation $h_1\in\u_f$ such that the continuation of the
finest dominated splitting of $f$ is not the finest one of $h_1$.
Then the finest dominated splitting of $h_1$ must be a refinement of
the continuation. Consider the subset $\u_{h_1}\cap \u_f$. If there
exists another perturbation $h_2$ in $\u_{h_1}\cap \u_f$ whose
finest dominated splitting is a refinement of the continuation of
the finest dominated splitting of $h_1$, we shift our attention to a
smaller subset $\u_{h_2}\cap\u_{h_1}\cap \u_f$. The process will
stop at some perturbation $h_n$ since the number of invariant
bundles in the finest dominated splitting of $h_i$ strictly
increases ( as $i$ increases ) and this number can not exceed $d$,
the dimension of $M$. Let $g=h_n$. Then $g$ must has robust finest
dominated splitting.

\qed

\bigskip

The main techniques in the proof of Theorem \ref{Thm:nonzeroexp} are
three results as we cite in the following:
\smallskip

A partially hyperbolic diffeomorphism $f$ is accessible if, for
every pair of points $p,\,\, q\in M$, there is a $C^1$ path from $p$
to $q$ whose tangent vector always lies in $E^u\cup E^s$ and
vanishes at most finitely many times. We say $f$ is stably
accessible if every g sufficiently $C^1-$close to $f$ is accessible.
\smallskip

\begin{Lem}[Main Theorem in \cite{DW}]\label{Lem:denseofstabacc}
For any $r\geq1$, stable accessibility is $C^1$ dense among the
$C^r$, partially hyperbolic diffeomorphisms of $M$, volume
preserving or not. 
\end{Lem}
\smallskip

\begin{Lem}[Thereom 2 in \cite{BB}]\label{Lem:BB}
Let $M$ be a compact manifold and $\omega$ be a smooth volume form.
Let $f$ be an $\omega-$preserving $C^1-$diffeomorphism of $M$,
admitting a dominated splitting $TM=E^1\oplus\cdots\oplus
E^k,\,k>1.$

Then there are $\omega-$preserving diffeomorphisms $g$, arbitrarily
$C^1-$close to $f$, for which the integral $\int_M log|Det
Dg|_{E^i}(x)|d\omega(x)$ is different from $0$ for each
$i\in\{1,...,k\}$.
\end{Lem}

\smallskip

Let $\mathcal{D}_p(f,m)$ be the set of points $x$ such that there is
an $m-$dominated splitting of index $p$ ( i.e., $p=dim E^s$ ) along
the orbit of $x$. Then $\mathcal{D}_p(f,m)$ is a closed set. Define
$$\Gamma_p(f,m)=M\backslash \mathcal{D}_p(f,m)$$
and
$$\Gamma_p(f,\infty)=\cap_{m\in\mathbb{N}}\Gamma_p(f,m).$$

We recall from Proposition \ref{Prop:OseleExterior} that for any
diffeomorphism $f\in \Diff^1_\omega(M)$ and $p\in\{1,...,d-1\}$, the
integrated Lyapunov exponent of $\wedge^p(Df)$ coincides with the sum
of the largest $p$ integrated Lyapunov exponents, that is,
$$\int_M\wedge^p(Df)d\omega=\int_{M}\wedge_p(f,x)d\omega(x),$$
where $\wedge_p(f,x)$ denotes the sum, i.e., $\wedge_p(f,x)=\lambda_1(f,x)+\cdots+\lambda_p(f,x)$.

\smallskip
\begin{Lem}[Proposition 4.17 in \cite{BV}]\label{Lem:glob}
Let $f\in \Diff^1_\omega(M)$ and $p\in\{1,...,d-1\}$, and given any
$\varepsilon_0>0$ and $\delta>0$, there exist a neighborhood
$\mathcal{U}(f,\varepsilon_0)\subset \Diff^1_\omega(M)$ of $f$ with
radius $\varepsilon_0$ and a diffeomorphism
$g\in\mathcal{U}(f,\varepsilon_0)$ such that
$$\int_{M}\wedge_p(g,x)d\omega(x)<\int_{M}\wedge_{p}(f,x)d\omega(x)-J_p(f)+\delta.$$
where $J_p(f)=\int_{\Gamma_p(f,\infty)}\frac{\lambda_p(f,x)-\lambda_{p+1}(f,x)}{2}d\omega(x).$
\end{Lem}
\bigskip


{\bf Proof of Theorem \ref{Thm:nonzeroexp}}\,\,

By Lemma\ref{Lem:denseofstabacc}, we can perturb $f$ to  a
diffeomorphism  $f_1\in\mathcal{PH}^1_\omega(M)$ with stable
accessibility. Take  a neighborhood $\u_1\subset
\mathcal{PH}^1_\omega(M)$ of $f_1$ such that each diffeomorphism  in
$\u_1$ is accessible. Applying Lemma\ref{Lem:finestdom}, we get
another small perturbation $f_2\in\u_1$ which is accessible and has
the robust finest dominated splitting: $$TM=E^s\oplus
E^{c1}\oplus\cdots\oplus E^{ck}\oplus E^u,\,\, 1\leq k\leq
d=dimM\eqno(2.1)$$ (Since any integrated Lyapunov exponent related
to $E^s$ and $E^u$ is robustly away from  zero, we need not care
about the decomposition of $E^s$ and $E^u$). Therefore, there is a
neighborhood $\u_2\subset\u_1$ of $f_2$ such that  any $g\in\u_2$,
$g$ is accessible and has the same finest dominated splitting as
$f_2$. By perturbing $f_2$ if necessary, Lemma\ref{Lem:BB} ensures
that there exists a neighborhood $\u_3\subset\u_2$ with property:
$$\int_M log|DetDg|_{E^{ci}}(x)|d\omega(x)\neq0,\quad1\leq i\leq k,\,\forall g\in\u_3.\eqno(2.2)$$
That is, the sum of the integrated Lyapunov exponents
$$\Sigma^{s+c_i}_{j=s+c_{i-1}+1}\int_M \lambda_j(g,x)d\omega(x)\neq0,
\quad 1\leq i\leq k$$ where we denote $s=dimE^s$ and $u=dimE^u$ and
$c_i=dim E^{ci}$ and $c_0=0$ for simplicity.

Notice that $$\int_M
\Lambda_p(g,x)d\omega(x)=\inf\limits_{n}\frac{1}{n}\int_M
log\|\Lambda^p(Dg^n)\|d\omega(x).$$ The functions
$g\in\Diff^1_\omega(M)\mapsto \int_M \Lambda_p(g,x)d\omega(x)$ are
upper semi-continuous for any $1\leq p\leq d.$ Hence the continuity
points of the map $$g\in\Diff^1_\omega(M)\mapsto (\int_M
\Lambda_1(g,x) d\omega(x),\cdots,\int_M
\Lambda_d(g,x)d\omega(x))\eqno(2.3)$$ form a residual subset. We
choose a continuity point $h$ of the above map in $\u_3$. Now we
verify that $h$ meets the requirements of our theorem.

Since $h\in\u_3\subset \u_1$, $h$ is accessible. This implies that
for $\omega-$almost every $x\in M$, the orbit of $x$ is dense in
$M$. In fact, Burns-Dolgopyat-Pesin\cite{BDP} pointed out in the proof of Theorem 2 that the essential accessibility property indicates that almost every point has a dense orbit. Note that the essential accessibility property is a weaker property than accessibility. Precisely, recall that accessibility
is an equivalence relation. If a diffeomorphism is
accessible then the partition into accessibility classes is trivial. A diffeomorphism is said to be essentially
accessible if the partition into accessibility classes is ergodic
(i.e. a measurable union of equivalence classes must have zero or full
measure). From the definition, one can deduce that an accessible diffeomorphism is also essentially accessible. Hence $h$ is essentially accessible, since $h$ is accessible. And then almost every point has a dense orbit.

Then for any $s+1\leq p\leq d-u$, we have either
$D_p(h,m)=M\,mod0$ for some integer $m>0$, or
$\Gamma_p(h,\infty)=M\,mod0$. In fact, if one has that $\omega (D_p(h,m))>0$ for some integer $m>0$, by the accessibility, for $\omega-$almost every point in $D_p(h,m)$, we can spread the dominated splitting along its orbit to the closure, that is, the whole manifold $M$. Since $h$ has the same finest
dominated splitting as $f_2$, we have that
$$\Gamma_p(h,\infty)=M\,mod0,\quad s+c_{i-1}+1\leq p\leq s+c_i,\,\forall 1\leq i\leq k$$ and $$D_p(h,m)=
M\,mod0,\quad p=c_i,\,\forall 1\leq i\leq k.$$

Note that $h$ is a continuity point of map (2.3). Combining with
Lemma\ref{Lem:glob}, one can obtain that $J_p(h)=0$, i.e.,
$$\lambda_p(h,x)=\lambda_{p+1}(h,x) \quad \forall s+c_{i-1}+1\leq
p\leq s+c_i,\,\forall 1\leq i\leq k.$$ Therefore, one can deduce from
(1.3) that
$$\int_M \lambda_p(h,x)d\omega(x)=\frac{1}{c_i}\int_M
log|DetDh|_{E^{ci}}(x)|d\omega(x)\neq0,\quad \forall s+c_{i-1}+1\leq
p\leq s+c_i,\,\forall 1\leq i\leq k. $$

\qed

\bigskip

{\bf{Remark}}\, When we have done this paper, we find \cite{HHTU}
which is considering the same problem in Theorem
\ref{Thm:nonzeroexp} of the case $dim E^c=2$.
\bigskip

\section{The example and Proofs of Theorems \ref{Thm:dichotomy}-\ref{Thm:interiorofHT}}
\bigskip
In this section we will first present an example of a linear Anosov
diffeomorphism $A$ on $\mathbb{T}^3$ which has a couple of complex
eigenvalues. Then we show that for any small neighborhood of $A$ in $\Diff^1_{\omega}(\mathbb{T}^3)$, there exists a residual subset consisting of diffeomorphisms with non-simple
spectrum. Then we will give the proof of Theorems \ref{Thm:dichotomy}, \ref{t.ergodicfarfromtangency} and \ref{Thm:interiorofHT}.
\bigskip

\subsection{The example}
\bigskip
{\bf Example:}
Let
$$
A=\left(\begin{array}{ccc}
0 &1 &0\\
0 &0 &1\\
1 &0 &1
\end{array}\right):\,\mathbb{T}^3\to\mathbb{T}^3,
$$
Then A is stable ergodic and
$$
\det(A-\lambda Id)=\left|\begin{array}{ccc}
-\lambda &1 &0\\
0 &-\lambda &1\\
1 &0 &1-\lambda
\end{array}\right|=\lambda^2(1-\lambda)+1=-\lambda^3+\lambda^2+1=f(\lambda).
$$
\smallskip
Since $f(1)=1>0$ and $f(2)=-3<0$, by the continuity of $f$, there is
$ c\in(1,\,2),$ such that $f(c)=0.$

On the other hand, one consider the converse matrix
$$
A^{-1}=\left(\begin{array}{ccc}
0 &-1 &1\\
1 &0 &0\\
0 &1 &0
\end{array}\right)\,\mbox{ and the determinant }\, \det(A^{-1}-\lambda
Id)=-\lambda^3-\lambda+1=g(\lambda).
$$
We have that $g(0)=1>0$ and $g(1)=-1<0$. This implies that $\exists
d\in(0,\,1),$ such that $g(d)=0.$ Note that
$g'(\lambda)=-3\lambda^2-1<0$.
Therefore, the point $d$ is the unique real  root of $g(\lambda)$
and hence the other two eigenvalues are not  real. Moreover,
observing that the coefficients of the first and last items are both
1, we can deduce that $d$ must be an irrational number.

Note that when a number $\lambda$ is an eigenvalue of $A$, its
reciprocal $\frac 1 \lambda$ is an eigenvalue of $A^{-1}$ and
conversely, it also holds. Thus we have the following conclusion
about $A$.

\smallskip

{\bf Conclusion I.\,}\, $\exists p\in\mathbb{T}^3$ (hyperbolic fixed
point) has complex eigenvalues $\lambda_1$, $\overline{\lambda_1}$
and a real (irrational) eigenvalue $\lambda_2$ satisfying
$$|\lambda_1|<1<|\lambda_2|.$$
\bigskip

This conclusion implies that $A$ has non-simple spectrum. Moreover, applying the following lemma, we will show that for any small neighborhood of $A$ in $\Diff^1_\omega(\mathbb{T}^3)$, there is a residual subset composed of diffeomorphisms with non-simple spectrum.

\smallskip

\begin{Lem}[Theorem 1 in \cite{BV}]\label{BV}
There exists a residual set $\r\subset \Diff^1_{\omega}(M)$ such that, for each $f\in\r$ and $\omega-$almost every $x\in M$, the Oseledec splitting of $f$ is either trivial or dominated at $x$.
\end{Lem}

\bigskip

We choose a neighborhood $\u\subset \Diff^1_{\omega}(\mathbb{T}^3)$ of
$A$, such that $\forall g\in\u$, $g$ is ergodic and has a fixed
point (i.e., the continuation of $p$ ) with complex eigenvalues. Using Lemma\ref{BV}, there
is a residual set $\cal{R}\subset\u$, such that
$\forall\,g\in\cal{R}$, the Oseledec decomposition of $g$ is
dominated or trivial. Therefore, $\forall\,g\in\cal{R}$, if $g$ has
simple spectrum, its dominated splitting must be the form
$$E_1(x)\bigoplus E_2(x)\bigoplus E_3(x)$$ for every generic point
$x\in O(\omega)$, the Oseledec basin of $\omega$, since $g$ is ergodic.
Note that $\omega$ is a volume form and the Oseledec basin has
$\omega-$full measure, the Oseledec basin is dense in $M$. Thus the
dominated splitting $E_1(x)\bigoplus E_2(x)\bigoplus E_3(x)$ of $g$
can be extended to the whole manifold $M$ by the continuity of
dominated splitting. This contradicts with the property that $g$ has
complex eigenvalues. Thus we have the second conclusion about the local genericity of diffeomorphisms with non-simple spectrum.

\smallskip

{\bf Conclusion II.\,}\,  For any small neighborhood $\v\subset \Diff^1_{\omega}(\mathbb{T}^3)$ of $A$, there exists a residual subset $\r\subset\v$ consisting of diffeomorphisms with non-simple
spectrum.

\bigskip

\subsection{The proof of Theorem \ref{Thm:dichotomy}}
\bigskip

Bonatti and Crovisier\cite{BC} proved that, generically, a
volume-preserving diffeomorphism is transitive in a compact,
connected manifold with a volume-preserving diffeomorphism.
\smallskip

\begin{Lem}[Thm1.3 in \cite{BC}]\label{Lem:Onehomclass}
Suppose M is connected. Then there is a residual subset $G_\omega$
in $\Diff^1_\omega(M)$ such that any $f\in G_\omega$ is transitive.
Moreover, $M$ is the unique homoclinic class.
\end{Lem}
\bigskip

To prove Theorem \ref{Thm:dichotomy}, we need introduce some
preliminary notions and lemmas first.
\smallskip

\begin{Def}\label{Def:linearsys}
We shall call any $4-$uple $\mathcal{A} = (\Sigma, f,\mathcal{E},A)$ to be a linear cocycle of dimension $d$
if:
\begin{enumerate}
\item[$\bullet$]$\Sigma$ is a set and $f:\Sigma\to\Sigma$ is a one-to-one map;
\item[$\bullet$]$\pi:\mathcal{E}\to\Sigma$ is a linear bundle of dimension $d$ over $\Sigma$, whose fibers are endowed with an Euclidean metric $\|\cdot\|$. The fiber over the point $x\in\Sigma$ will be denoted by $\mathcal{E}_x$;
\item[$\bullet$] $A : x\in\Sigma\mapsto A_x\in GL(\mathcal{E}_x,\mathcal{E}_{f(x)})$ is a map.
\end{enumerate}
We shall say that a linear cocycle $A$ is bounded if there exists a constant $K > 0$ such that,
for any $x\in\Sigma$, we have $\|A_x\| < K$ and $\|A^{-1}x\|< K$. We call $K$ a bound of $\mathcal{A}$ .
\end{Def}

{\bf Remark} Here the base space $\Sigma$ can be taken arbitrarily, even not a topology space.

Let $(\Sigma, f, \mathcal{E},A)$ be a linear system, an invariant
subbundle is a collection of linear subspaces $F(x) \subset \mathcal{E}_x$ whose dimensions do
not depend on $x$ and such that $A(F(x)) = F(f(x))$. An $A-$invariant splitting
$F\oplus G$ is given by two invariant subbundles such that $\mathcal{E}_x = F(x)\oplus G(x)$ at
each $x\in\Sigma$.

\begin{Def} Let $(\Sigma, f, \mathcal{E},A)$ be a linear system and $\mathcal{E} = F\oplus G$ an $A-$invariant
splitting. We say that $F\oplus G$ is a dominated splitting if there exists
$n\in\mathbb{N}$ such that
$$\|A^{(n)}(x)|_F\|\|A^{(-n)}(f^n(x))|_G\| < 1/2$$
for every $x\in\Sigma$. We write $F\prec G$.
\end{Def}

If we want to emphasize the role of n then we say that $F\oplus G$
is an $n-$dominated splitting and write $F\prec_n G$. Finally, the
dimension of the dominated splitting is the dimension of the
subbundle $F$.

\smallskip

\begin{Def}\label{rank}

Let $A$ be a linear map on a $d-$dimensional linear space. The
complex eigenvalues $(\lambda,\,\bar{\lambda})$ of $A$ is called of
rank $(i,\,i+1)$, where $1\leq i\leq d-1$, if the moduli of all its
other eigenvalues are different from $|\lambda|$ and the number of
the eigenvalues which are less than $|\lambda|$ coincide with $i-1$.
\end{Def}
\smallskip

The following is Theorem 3.6 in \cite{LL} with more details.

\begin{Lem}[Theorem 3.6 in \cite{LL}]\label{Thm:complexeigenvalue}
Given $K>0$ and $\varepsilon>0$ there is $\ell\in\mathbb{N}$ such that for any linear periodic system $\mathcal{A} = (\Sigma, f, \mathcal{E},A)$ bounded by $K$, one has that it admits a finest dominated splitting $E = E_1\oplus_{\prec_\ell}
E_2
\oplus_{\prec_\ell}\cdots\oplus_{\prec_\ell} E_k$
if and only if we can not get a complex eigenvalue of the following ranks
$$(\tau_1,\tau_1 + 1), (\tau_2,\tau_2 + 1), \cdots , (\tau_{k-1}, \tau_{k-1} + 1)$$
by any $\varepsilon-$perturbation of $\mathcal{A}$ , where $\tau_i =
\sum^i_{j=1} dim E_j , i = 1, 2,\cdots, k.$
\end{Lem}

\bigskip

{\bf Proof of Theorem \ref{Thm:dichotomy}:}\,\, Let $G_\omega$ be
the residual subset of $\Diff^1_\omega(M)$ determined by Lemma
\ref{Lem:Onehomclass}. Define $$\mathcal{R}=G_\omega\cap
\mathcal{PH}^1_\omega(M).$$ Then $\mathcal{R}$ is a residual subset $
\mathcal{PH}^1_\omega(M)$ of diffeomorphisms $f$ such that 
$M$ is the unique homoclinic class.

For any diffeomorphism $f\in\mathcal{R}$, we take $K=\max_{x\in M}\{\|Df|_{E^c}\|,\|Df^{-1}|_{E^c}\|\}$.
Then we obtain a positive integer $\ell$ by Lemma \ref{Thm:complexeigenvalue}.
If $E^c$ admits an $\ell-$finest dominated splitting as the form (1.5), the proof is done. Otherwise,
there is an integer $s<i<s+c$, such that $E^c(x)$ has no dominated decomposition $E^c(x)=E(x)\oplus_\prec F(x)$ with
$dim E=i-s$ for $x$ in some positive-measure subset of $M$.
Then there are
periodic points such that the cocycles defined over these points 
do not have dominated
decomposition $E^c=E\oplus_{\prec_n} F$ with $dim E=i-s$ and uniform time $n$,  for any integer $n\in \mathbb{N}$. 
Let $$ \Sigma=M,\,\mathcal{E}=E^c,\,A=Df|_{E^c}.$$ Then
$\mathcal{A}=(\Sigma,f,\mathcal{E},A)$ is a linear cocycle bounded
by $K$ over an infinite periodic system having transitions but without dominated splitting $E^c=E\oplus_\prec F$ where $dim E=i-s$. Now
using Lemma \ref{Thm:complexeigenvalue}, for any $\varepsilon>0$,
we obtain a $\varepsilon-$pertubation $B$ of $A$ and a periodic
point in $\Sigma$, such that $M_{q,B}=B(f^{p(q)-1}q)\circ\cdots\circ
B(q)$ has a complex eigenvalue of rank $(i,i+1)$. Applying Franks'
Lemma, we get a $C^1-$perturbation $g\in \mathcal{PH}^1_\omega(M)$ of $f$,
which coincides with $f$ out of an arbitrarily small neighborhood of
$Orb(q)$, equals to $f$ in $Orb(q)$, and whose derivation satisfying
$$Dg|_{E^c(g,f^iq)}=B_{f^iq},\,i=0,1,...,p(q)-1.$$
Hence $Dg^{p(q)}_q$ has a complex eigenvalue.

\qed
\bigskip

\subsection{The proof of Theorem \ref{t.ergodicfarfromtangency}}
\smallskip

In this subsection, we focus on conservative diffeomorphisms far from homoclinic tangencies and discuss their two generic properties: one is the property of admitting a dominated splitting as (1.5) (Lemma \ref{Lem:farfromtangency}) and the other is ergodicity (Lemma \ref{l.denseergodic}).

\begin{Lem}\label{Lem:farfromtangency}
For $C^1$ generic $f\in \Diff^1_\omega(M)$, $(d=dim M\geq3)$, it is
$C^1$ far away from tangencies if and only if there exists a
dominated splitting $E^s\oplus E^c\oplus E^u$ with two non-trivial
extreme subbundles and the finest dominated splitting of the center
bundle is of the form
$$E^c(x)=\oplus^c_{i=1}E^{ci}(x)\mbox{ with }dim E^{ci}(x)=1,\mbox{
for all }i\in\{1,2,...,c\},\eqno(1.6)$$ $\omega-a.e.x\in M,$ where
$c=dim E^c$.
\end{Lem}
\bigskip

Abdenur-Bonatti-Crovisier-Diaz-Wen claim that, for $C^1$-generic diffeomorphisms, the set of indices of the (hyperbolic) periodic points in a chain recurrence class (in fact, such classes are homoclinic ones) form an
interval in $\mathbb{N}$.
Applying the connecting lemma and Franks' lemma for conservative diffeomorphisms, we obtain the conservative version of this result (Theorem 1.1 in \cite{ABCDW}) as follows:

\begin{Lem}[the conservative version of Thm 1.1 in \cite{ABCDW}]\label{Lem:ABCDW}
There is a residual subset $I$ of $\Diff_\omega^1(M)$ of
diffeomorphisms $f$ such that, for every $f\in I$, any homoclinic
class $H(p, f)$ containing hyperbolic saddles of indices $\alpha$
and $\beta$ contains a dense subset of saddles of index $\tau$ for
all $\tau\in [\alpha, \beta]\cap\mathbb{N}$.
\end{Lem}

\bigskip

Recall that a point $x\in M$ is called $(C^1)$ $i-$preperiodic of $f$, $0\leq i\leq d$, if for any neighborhood
$\mathcal{U}$ of $f$ in $\Diff^1(M)$ and any neighborhood $U$ of $x$ in $M$, there exist $g\in \mathcal{U}$ and $y\in U$ such
that $y$ is a hyperbolic periodic point of $g$ with index $i$. 
Now we begin to prove one of the main lemma in this subsection:

\bigskip

{\bf Proof of Lemma \ref{Lem:farfromtangency}:}\,\, Let
$$\r=G_\omega\cap I,$$ where $G_\omega$ and $I$ are determined
by Lemma \ref{Lem:Onehomclass} and \ref{Lem:ABCDW}, respectively. 
For $f\in \r$, let $i_0$(resp $i_1$) be the minimal (resp. maximal)
preperiodic index of $f$. Then $i_0\geq 0$ and $i_1\leq d$. By the
definition of preperiodic points, we can assume that $f$ itself
contains index $i_0,i_1$ periodic points. By Lemma \ref{Lem:ABCDW},
$f$ contains $i-$index periodic points for all $i\in [i_0,
i_1]\cap\mathbb{N}$. Since $M$ is the unique homoclinic class by
Lemma \ref{Lem:Onehomclass} and $f$ is far from tangencies, now we obtain a dominated splitting
with center bundles all 1 dimensional by the equivalent conditions for the existence of dominated splitting(for more details, see \cite{LLS, W}).

Since $f$ is conservative, one must have that $i_0\geq1$ and
$i_1\leq d-1$. Otherwise, if $i_0=0$, the minimal of periodic index
can be 0 or 1(There may exist a weak stable subbundle with dimension
1). In the first case, all the Lyapunov exponents of $f$ are
non-negative. Hence the sum of the Lyapunov exponents are nonzero.
This contradicts with the conservative condition. In the second
case, the negative exponent have multiplicity 1 and is close to 0,
but the positive exponents are uniformly far from 0 by the
domination. This again deduces the same contradiction as the first
case. The proof of $i_1\leq d-1$ is analogous.

The rest thing is to prove the two extreme bundles are uniformly
hyperbolic. This can be obtained by showing that it is forbidden to
decrease the index of a periodic point with index $i_0$ or increase
that of a periodic point with index $i_1$ by perturbation. The proof
essentially from $Ma\tilde{n}\acute{e}$'s Ergodic Closing Lemma
\cite{M}, in conservative version (see also \cite{W2} or Theorem B
and Section 4 in \cite{BoDPR} ). This concludes the proof of
necessity.

The sufficiency is ensured by the uniform hyperbolicity of the
extreme subbundles and the domination of the splitting (1.6) of the
center bundle. This ends the whole proof.

\qed

\bigskip

{\bf Remark:}\,\,Dawei Yang told us the analogous consequence for dissipative dynamics of Lemma
\ref{Lem:farfromtangency} may be deduced by his recent joint
work with Crovisier and Sambrino. And he suggested us to omit the
original hypothesis of partial hyperbolicity.

\bigskip

Recall that the integers $s$, $c$ and $u$ denote the dimensions of $f-$invariant subbundles $E^s$, $E^c$ and $E^u$ in the dominated splitting $TM=E^s\oplus E^c\oplus E^u$. To emphasis their dependence on $f$, we write them as $s(f)$, $c(f)$ and $u(f)$. Since $(G1)-(G4)$ in the following proposition are open properties, we combine the preceding lemma with Lemma \ref{Lem:ABCDW}, \ref{Lem:denseofstabacc} and Theorem \ref{Thm:nonzeroexp} and obtain that:

\begin{Prop}\label{p.finestsplittingfarfromtangencies} There is a $C^1$ open and dense subset $\mathcal{O}$ of the volume preserving diffeomorphsims far from tangencies, such that any $f\in \mathcal{O}$ satisfies the following properties.
\begin{itemize}
\item[(G1)] $f$ is partially hyperbolic, it admits a partially hyperbolic splitting admits a finest dominated splitting
$E^s\oplus E^{c1}\oplus \cdots \oplus E^{c,c(f)}\oplus E^u$, where $dim(E^{ci})=1$.

\item[(G2)] $f$ has periodic points with index $dim(E^s),dim(E^s)+1,\cdots,dim(E^s)+c(f)$.

\item[(G3)] $f$ is accessible.

\item[(G4)] There is $0\leq j(f) \leq c(f)$, such that for any integer $1\leq j\leq j(f)$, one has
$$\int \log \|Df|_{E^{c,j}(x)}\|d\omega(x)<0,$$ and for any integer $j(f)+1\leq j\leq c(f)$, we has $$\int \log \|Df|_{E^{c,j}(x)}\|d\omega(x)>0.$$
When $j(f)=0$ (resp. $j(f)=c(f)$), we take $E^{c,j(f)}(x)$ (resp. $E^{c,j(f)+1}(x)$) to be vanished.

\end{itemize}
\end{Prop}
\bigskip

Recently, F.R.Hertz, M.A.R.Hertz, A.Tahzibi, and R.Ures\cite{HHTU, HHTU2} developed a new criteria of ergodicity and nonuniform hyperbolicity which provided fresh ideas to the Pugh-Shub stable ergodicity conjecture.
Applying Proposition \ref{p.finestsplittingfarfromtangencies} and their criteria, we illustrate the density of ergodicity as the follow which implies the ergodic diffeomorphisms form a residual subset, since it is a $G_\delta$ set.

\begin{Lem}\label{l.denseergodic}
There is a $C^1$ dense subset of diffeomorphisms $\mathcal{E}$ in $\mathcal{O}$, such that any diffeomorphism
in $\mathcal{E}$ is ergodic.
\end{Lem}

\bigskip

Before give the proof of Lemma~\ref{l.denseergodic}, we need give some definitions and notations.
\smallskip

\subsubsection*{$\bullet$ Existence of stable manifolds}

For each point $x$ in a compact Riemannian manifold $M$,the Pesin stable manifold
of $x$ is
$$W^s(x)=\{ y\in M : \limsup\limits_{n\rightarrow \infty}\frac{1}{n} \log d(f^n(x),f^n(y))<0\}$$
and the Pesin unstable manifold of $x$, $W^u(x)$, is the Pesin stable manifold for $f^{-1}$.
In the context, we use 
$\mathcal{W}^s$(resp. $\mathcal{W}^u$) to denote the strong stable (resp. unstable) foliation. 
\smallskip

For any $i\in\mathbb{N}$, let $I^i_1=(-1,1)^i$ and $I^i_\varepsilon= (-\varepsilon;\varepsilon)^i$ and denote by $Emb^1(I^i_1,M)$ the set of $C^1$-embeddings of $I^i_1$ on $M$.
Suppose a compact invariant set $\Lambda$ admits a dominated splitting $E\oplus F$. 
The following lemmas are Lemma 3.0.4 and Corollary 3.3 in \cite{PS} for $C^1$ case with high dimension which come from the existence of the dominated splitting on $\Lambda$.

\begin{Lem}[Lemma 3.0.4 in \cite{PS}]\label{l.stablemanifold}
There exist two continuous functions $$\Phi^{cs}: \Lambda\longrightarrow Emb^1(I^{dimE}_1,M)\quad\mbox{ and }\quad\Phi^{cu}: \Lambda\longrightarrow Emb^1(I^{dimF}_1,M)$$ such that, with $W^{cs}_\varepsilon(x)=\Phi^{cs}(x)I^{dimE}_\varepsilon$ and $W^{cu}_\varepsilon(x)=\Phi^{cu}(x)I^{dimF}_\varepsilon$, the following properties hold:
\begin{itemize}
\item[(a)] $T_xW^{cs}_\varepsilon(x)=E(x)$ and $T_xW^{cu}_\varepsilon(x)=F(x)$;
\item[(b)] for any $0<\varepsilon_1<1$, there exists $\varepsilon_2$ such that $f(W^{cs}_{\varepsilon_2}(x))\subset W^{cs}_{\varepsilon_1}(f(x))$ and $f^{-1}(W^{cu}_{\varepsilon_2}(x))\subset W^{cu}_{\varepsilon_1}(f^{-1}x)$;
\end{itemize}
\end{Lem}

\begin{Cor}[Corollary 3.3 in \cite{PS}]\label{c.stablemanifold}
For $0<\lambda<1$, there exists $\varepsilon>0$ such that, for any $x\in\Lambda$ satisfying $\prod\limits_{j=0}^{n-1}\left\|Df|_{E(f^j(x))}\right\|\leq \lambda^n$ for all $n>0$, one has $diam(f^n(W^{cs}_\varepsilon(x)))\longrightarrow 0$.
\end{Cor}

As a corollary of Lemma~\ref{l.stablemanifold} and Corollary~\ref{c.stablemanifold}, we have the following proposition:

\begin{Prop}\label{p.stablemanifoldtheorem}
Let $f$ be a volume preserving diffeomorphism. Suppose $f$ admits a dominated splitting $E^{cs}\oplus E^c_1\oplus E^{cu}$ with $dim(E^c_1)=1$ and the set $\Lambda^s=\{x; \lambda^c(x)<0\}$ has positive volume. Then Lebesgue almost every point $x\in \Lambda^s$ has a local stable manifold $W^s_{loc}(x)$ with dimension $E^{cs}\oplus E^c_1$. Moreover, $W^s_{loc}(x)$ is tangent to the bundle $E^{cs}\oplus E^c_1$ (this means, then the tangent space of $W^s_{loc}(x)$ is contained in a small cone around $E^{cs}\oplus E^c_1$).
\end{Prop}
\bigskip

\begin{Rem}\label{r.stablemanifold}
Suppose $\Lambda$ admits a partially hyperbolic splitting $E^s\oplus E^{cs}\oplus E^c\oplus E^u$ and there is a local strong stable leaf $\mathcal{W}^s_{loc}(x)$ tangent to $E^s(x)$ at some point $x$. If 
the center Lyapunov exponent at $x$ is negative(resp. positive), then $\mathcal{W}^s_{loc}(x)\subset W^{s}_\varepsilon(x)$(resp. $\mathcal{W}^u_{loc}(x)\subset W^{u}_\varepsilon(x)$).
\end{Rem}

\bigskip

\subsubsection*{$\bullet$ Blenders}

Topologically blenders are Cantor sets with distinctive geometric feature which is introduced by C. Bonatti and L.J.Diaz in \cite{BD} to give a slightly different mechanism for constructing non-Axiom A
diffeomorphisms and robust heterodimensional cycles. Later, the notion of blender was motivated by L.J.Diaz in \cite{D} and developed by F.Herz-M.A.Hertz-A.Tahzibi-R.Ures in \cite{HHTU} to produce a local source of stable ergodicity.
\smallskip

\begin{Def}\label{d.blender}Let $q,p$ be hyperbolic periodic points of a diffeomorphism $f$ with stable index $i$ and $i+1$ respectively.

We say that $f$ has a $cs-$blender of index $i$ associated to ($q,p$) if there is a $C^1$ neighborhood $\mathcal{U}$ of $f$ such that, for every $g\in\mathcal{U}$, one has $W^s(p_g)$ is contained in the closure of $W^s(q_g)$, where $q_g,p_g$ are the analytic continuation of $q$ and $p$ for $g$.

Define a $cu-$blender in an analogous way, by concerning $f^{-1}$.
\end{Def}

There are several definitions of blender given in \cite{BD}, \cite{BDV}, \cite{D} and \cite{HHTU2}. Our definition comes from the Definition and Lemma~1.10 in page~369 of \cite{BD}. And the proposition bellow which partially comes from the main result of \cite{HHTU2} holds for our definition.

\begin{Prop}\label{p.blender}
Let $f : M\longrightarrow M$ be a diffeomorphism admitting a dominated splitting $E^{cs}\oplus E^c_1\oplus E^{cu}$ with $dim(E^{cs})=i>0$ and $dim(E^c_1)=1$. Suppose $f$ has a $cs-$blender of index $i$ associated to ($q,p$). Then there is a $C^1$ neighborhood $\mathcal{U}$ of $f$, such that for any $g\in \mathcal U$, and for every $i+1-$dimension disk $D$ which is tangent to a cone around the bundle $E^{cu}\oplus E^c_1$, we have that $$D\pitchfork W^s(p_g)\neq \phi\quad\mbox{ implies }\quad D\pitchfork W^s(q_g)\neq \phi,$$ where $q_g,p_g$ are the analytic continuation of $q$ and $p$ for $g$.
\end{Prop}
\begin{proof}
We only need to notice that $W^s(q)$ is tangent to the bundle $E^{cs}$. From the definition of $cs-$blender, one can conclude the proof.
\end{proof}

\begin{Def}A diffeomorphism $f$ is called to has a chain of $cs-$blenders of index $(i_0,\cdots,i_1)$ if
\begin{itemize}
\item[(a)] for any $i_0\leq i \leq i_1$, $f$ has a $cs-$blender of index $i$
associated to periodic points $(q_i,p_i)$,
\item[(b)] for any $i_0+1\leq i \leq i_1$, $q_i$ is homoclinically related to $p_{i-1}$.
\end{itemize}
\end{Def}

It is remarkable that the distinctive blender property is a robust property. And homoclinical relation is also an open property. Hence, if a diffeomorphism $f$ has a chain of blenders of index ($i_0,\cdots,i_1$), there is a neighborhood $\mathcal{U}$ of $f$, such that any $g\in \mathcal{U}$ has a chain of blenders of index ($i_0,\cdots,i_1$) as well.
\smallskip

Connecting lemma was proved by Hayashi \cite{H} at first, and
then was extended to the conservative setting by Wen and Xia \cite{WX, W3}(see also M.-C. Arnaud\cite{A}). Later, from the proof of Hayashi's Connecting Lemma, Bonatti and Crovisier\cite{BC} extract a slightly stronger statement in the next lemma which permits to perform dynamically perturbations and create intersections between stable and unstable bundles.

\begin{Lem}[Theorem 2.1(Connecting lemma) in \cite{BC}, Theorem 3.10 in \cite{HHTU2}]\label{Lem:connectinglem}
Let $p$, $q$ be hyperbolic periodic points of a $C^r$ transitive diffeomorphism preserving a smooth measure $m$. Then, there exists a $C^1-$perturbation $g\in C^r$preserving $m$ such that $W^s(p)\cap W^u(q)\neq\emptyset$.
\end{Lem}
\smallskip

Lemma \ref{Lem:Onehomclass} ensures the transitivity of the diffeomorphisms we are considering. Combining with the previous lemma, we can prove the following proposition.

\begin{Prop}\label{p.blenderfarfromtangencies}
There is a $C^1$ open and dense subset $\mathcal{O}_1\subset \mathcal{O}$ of diffeomorphisms such that for any $f\in \mathcal{O}_1$, it admits a dominated splitting $E^s\oplus E^{c1}\oplus \cdots \oplus E^{c,c(f)}\oplus E^u$. Moreover, $f$ has a chain of $cs-$blenders of index ($dim(E^s),\cdots,dim(E^s)+c(f)-1$).
\end{Prop}
\smallskip

{\bf Proof:} 
Take $f\in \mathcal{O}$, then $f$ has hyperbolic periodic points with index $dim(E^s)$ and $dim(E^s)+1$. By Proposition~\ref{p.blender}, there is an open set $\mathcal{U}_1\subset \mathcal{O}$ arbitrarily close to $f$, such that any $g\in \mathcal{U}_1$ has a $cs-$blender of index $dim(E^s)$ associated to $(q_{1,g},p_{1,g})$, where $q_{1,g}$ and $p_{1,g}$ denote the continuation of $q_1$ and $p_1$ for $g$. Note that $\mathcal{O}$ is an open set. Fix a diffeomorphism $f_1\in\mathcal{U}_1$ and by the same argument, we can find another open set $\mathcal{U}_2\subset\mathcal{U}_1$ such that any $g\in \mathcal{U}_1$ has a $cs-$blender of index $dim(E^s)$ associated to $(q_{2,g},p_{2,g})$. Inductively, we obtain open sets $\mathcal{U}_{c(f)}\subset\mathcal{U}_{c(f)-1}\subset\cdots \mathcal{U}_1$ such that for any $1\leq i \leq c(f)$ and any $g\in \mathcal{U}_i$, $g$ has a $cs-$blender of index $dim(E^s)+i-1$ associated to ($q_{i,g},p_{i,g}$).

Take $g\in \mathcal{U}_{c(f)}$, by Lemma \ref{Lem:Onehomclass} and using Lemma \ref{Lem:connectinglem} twice, we get a diffeomorphism $g_1\in \mathcal{U}_{c(f)}$ which is arbitrarily close to $g$, such that $p_{1,g_1}$ and $q_{2,g_1}$ are homoclinically related to each other. Do such perturbation $c(f)-1$ times, we obtain a diffeomorphism $g_{c(f)-1}\in \mathcal{U}_{c(f)}$ such that for any $1\leq i \leq c(f)-1$, one has that $p_{i,g_{c(f)-1}}$ and $q_{i+1,g_{c(f)-1}}$ are homoclinically related.

Let $\mathcal{O}_1$ be the subset consisting of all the diffeomorphisms in $\mathcal{O}$ with chains of $cs-$blenders of index ($dim(E^s),\cdots,dim(E^s)+c(f)-1$). Then we verified the density of $\mathcal{O}_1$.
Since homoclinic relation is an open property, it must be an open and dense set as stated in the theorem.
\qed
\bigskip

\begin{Prop}\label{p.chainblender}
Suppose $f$ admits a dominated splitting $E^s\oplus E^{c1}\oplus \cdots \oplus E^{cc}\oplus E^u$ with $dim(E^{ci})=1$ for any $1\leq i \leq c$, and has a chain of blenders $\{(q_i,p_i)\}_{i=1}^{c(f)}$ of index $(dim(E^s),\cdots,$ $dim(E^s)+c-1)$. Then for any $(d-(dim(E^s)+j))-$dimension disk $D$ which is tangent to a cone field around $E^{c,j+1}\oplus\cdots \oplus E^{c,c(f)}\oplus E^u$ and satisfies $D\pitchfork W^s(p_{c(f)})\neq \emptyset$, $1\leq j \leq c(f)$, it holds $$D\pitchfork W^s(q_{j+1})\neq \emptyset \quad\mbox{ and }\quad D\pitchfork W^s(p_{j})\neq \emptyset.$$
\end{Prop}

\smallskip

{\bf Proof:} 
It suffices to show that, there is $n>0$ such that $$f^{n}(D)\pitchfork W^s(q_{j+1})\neq \emptyset\quad\mbox{ and }\quad f^{n}(D)\pitchfork W^s(p_{j})\neq\emptyset.$$ We will only prove the first part, since the second part comes analogously.

For any cone field around $E^{c,j+1}\oplus \cdots \oplus E^{c,c(f)}\oplus E^u$, note that $Df$ preserves this cone field and contracts its area uniformly. Thus,
if $D$ is tangent to a bundle which is tangent to such a cone field
, then so do $f^n(D)$.

We prove this by induction. Since $D\pitchfork W^s(p_{c(f)})\neq \emptyset$, there exists a submanifold $D_0$ in $D$ with dimension $dim(E^u)+1$, which is tangent to a cone field around $E^{c,c(f)}\oplus E^u$ and satisfies $D_0\pitchfork W^s(p_{c(f)})\neq \emptyset$. By Proposition \ref{p.blender}, we have $$D_0\pitchfork W^s(q_{c(f)})\neq \emptyset.$$ Thus by $\lambda$-lemma, there is an integer $n_0>0$, such that $f^{n_0}(D_0)\pitchfork W^s(p_{c(f)-1})\neq \emptyset$, since $q_{c(f)}$ and $p_{c(f)-1}$ are homoclinically related. Then $f^{n_0}(D)\pitchfork W^s(p_{c(f)-1})\neq \emptyset$. There is a submanifold $D_1\subset D$ with dimension $dim(E^u)+2$, such that $D_1$ is tangent to a cone around $E^{c,c(f)-1}\oplus E^{c,c(f)}\oplus E^u$, and $D_1\pitchfork W^s(p_{c(f)-1})\neq \emptyset$. Continue the argument and then we can complete the proof.
\qed
\bigskip

\subsubsection*{$\bullet$ New Criteria of ergodicity}

Let $p$ be a saddle of a $C^{1+\alpha}$ diffeomorphism $f$ and $Orb(p)$ denote the orbit of $p$.
We write $$B^s(p,f)=\{x: W^s(Orb(p))\pitchfork W^u(x)\neq \phi\},$$
 $$B^u(p,f)=\{x: W^u(Orb(p))\pitchfork W^s(x)\neq \phi\}.$$

The following proposition is a direct corollary of Theorem~A of \cite{HHTU} which presents a new criteria of ergodicity.

\begin{Prop}\label{p.criteria}
Let $f: M\longrightarrow M$ be a $C^{1+\alpha}$ diffeomorphism preserving the volume measure. Suppose that $p$ is a saddle of $f$ such that $$\omega(B^s(p,f)),\omega(B^u(p,f))>0\quad\mbox{ and }\quad \omega(M\setminus (B^s(p,f)\cup B^u(p,f)))=0.$$ Then $f$ is ergodic.
\end{Prop}
\smallskip

Together with the following result of Avila, we will prove Lemma \ref{l.denseergodic} and then complete the proof of Theorem \ref{t.ergodicfarfromtangency}.

\begin{Lem}(Theorem 1 in \cite{AV})\label{Lem:denseofC2}\,\,
$C^\infty$ diffeomorphisms are dense in $\Diff^1_\omega(M)$.
\end{Lem}
\bigskip

{\bf Proof of Lemma~\ref{l.denseergodic}}
Recall that the set $\mathcal{O}_1\subset \mathcal{O}$ is given in Proposition~\ref{p.blenderfarfromtangencies}. We will use Proposition~\ref{p.criteria} to prove that any $C^{1+\alpha}$ volume preserving diffeomorphism $f\in \mathcal{O}_1$ is ergodic. Since $C^\infty$ volume preserving diffeomorphisms are dense in the $C^1$ topology by Lemma \ref{Lem:denseofC2}, we then can conclude that ergodic diffeomorphisms are dense in $\mathcal{O}_1$.

Combining Proposition~\ref{p.finestsplittingfarfromtangencies} and Proposition~\ref{p.blenderfarfromtangencies}, we know that $f\in \mathcal{O}_1 $ satisfies the following properties.
\begin{itemize}
\item[(G1)] $f$ is partially hyperbolic, it admits a partially hyperbolic splitting admits a finest dominated splitting
$E^s\oplus E^{c1}\oplus \cdots \oplus E^{c,c(f)}\oplus E^u$, where $dim(E^{ci})=1$.

\item[(G2')] $f$ has a chain of $cs-$blenders of index $(dim(E^s),\cdots,dim(E^s)+c(f)-1)$ and a chain of $cu-$blenders of index $(dim(E^s)+1,\cdots,dim(E^s)+c(f))$.

\item[(G3)] $f$ is accessible.

\item[(G4)] There is $0\leq j(f) \leq c(f)$, such that for any integer $1\leq j\leq j(f)$, one has
$$\int \log \|Df|_{E^{c,j}(x)}\|d\omega(x)<0,$$ and for any integer $j(f)+1\leq j\leq c(f)$, we has $$\int \log \|Df|_{E^{c,j}(x)}\|d\omega(x)>0.$$
When $j(f)=0$ (resp. $j(f)=c(f)$), we take $E^{c,j(f)}(x)$ (resp. $E^{c,j(f)+1}(x)$) to be vanished.

\end{itemize}

At first, we consider two special cases: $j(f)=0$ and $j(f)=c(f)$. One can show the ergodicity of $f$ by a theorem of Burns-Dolgopyat-Pesin \cite{BDP}.

\begin{Lem}[Theorem 4 in \cite{BDP}]\label{Lem:BDP-ergodic} Let $f$ be a $C^{1+\alpha}$ partially hyperbolic diffeomorphism preserving a smooth measure. Assume that
$f$ is accessible and has negative center Lyapunov exponents on a set of positive
measure. Then $f$ is stably ergodic.
\end{Lem}

From now on, we assume that $0<j(f)<c(f)$. Denote $\lambda^c_i(x)$ ($1\leq i \leq c(f)$) to be the center Lyapunov exponent of $x$ along the center bundle $E^{ci}$. And denote $$\Lambda^{cs}=\{x|\lambda^c_{j(f)}(x)<0\}\quad\mbox{ and }\quad\Lambda^{cu}=\{x|\lambda^c_{j(f)+1}(x)>0\}.$$
Because $$\int \log \|Df|_{E^{c,j(f)}(x)}\|d\omega(x)<0\quad\mbox{ and }\quad\int \log \|Df|_{E^{c,j(f)+1}(x)}\|d\omega(x)>0,$$ one has $\omega(\Lambda^{cs})>0$, $\omega(\Lambda^{cu})>0$. Moreover, for $\omega-a.e.\,\,x$, note that the subbundle $E^{c,j(f)}(x)$ is dominated by the other bundle $E^{c,j(f)+1}(x)$ and hence at least one of the Lyapunov exponents $\lambda^c_{j(f)}(x)$ and $\lambda^c_{j(f)+1}(x)$ should be nonzero. Thus we have that $\Lambda^{cs}\cup \Lambda^{cu}$ has full volume.

Recall that $f$ has a chain of $cs-$blenders of index $(dim(E^s),\cdots,dim(E^s)+c(f)-1)$ by $(G2')$, denoted by $\{q_{i},p_{i}\}_{i=1}^{c(f)}$. Then for $i=j(f)$, $f$ has a $cs-$blender of index $dim(E^s)+j(f)-1$ associated to ($q_{j(f)},p_{j(f)}$). By Proposition~\ref{p.criteria}, Lemma~\ref{l.denseergodic} is a corollary of the following facts: $\omega-$almost every point of $\Lambda^{cu}$ belongs to $B^s(p_{j(f)},f)$ and $\omega-$almost every point of $\Lambda^{cs}$ belongs to $B^u(p_{j(f)},f)$. We only prove the first part and the proof of the second part is similar.

By Proposition~\ref{p.stablemanifoldtheorem}, $\omega-$almost every point $x$ in $\Lambda^{cu}$ has local unstable manifold $W^{u}_{loc}(x)$ of dimension $d-(dim(E^{cs})+j(f))$ and moreover, $W^{u}_{loc}(x)$ is tangent to the bundle $E^{c,j(f)+1}\oplus\cdots\oplus E^{c,c(f)}\oplus E^u$.

 Note that $f$ has a $cs-$blender of index $dim(E^s)+c(f)-1$ associated to $(q_{c(f)},p_{c(f)})$. Because $f$ is accessible, the orbit of almost every point is dense in the ambient manifold $M$. We can assume that the point $x$ is arbitrarily close to the periodic points $p_{c(f)}$ of index $dim(E^s)+c(f)$.
Thus the strong unstable leaf at $x$ should intersect the stable manifold at $p_{c(f)}$ transversely, i.e., $\mathcal{W}^u(x)\pitchfork W^s(p_{c(f)})\neq \emptyset$. Since $x\in\Lambda^{cu}$, by Remark~\ref{r.stablemanifold}, we have that $\mathcal{W}^u(x)\subset W^u_{loc}(x)$. Then $W^u_{loc}(x)\pitchfork W^s(p_{c(f)})\neq \emptyset$.
Now it follows from Proposition~\ref{p.chainblender} that $W^u_{loc}(x)\pitchfork W^s(p_{j(f)})\neq \emptyset$ by taking $D=W^u_{loc}(x)$. Hence $x\in B^s(p_{j(f)})$ and we conclude the proof.
\qed
\bigskip

{\bf Proof of Theorem \ref{t.ergodicfarfromtangency}}

Since ergodic diffeomorphisms form a $G_\delta$ set, Lemma \ref{l.denseergodic} implies that there is a generic subset $\mathcal{E}$ among the set of volume preserving diffeomophisms far from tangencies, such that any diffeomorphism $f\in \mathcal{E}$ is ergodic.

We take $\mathcal{R}=\mathcal{O}\cap \mathcal{E}$ and, together with Proposition \ref{p.finestsplittingfarfromtangencies}, conclude the proof.
\qed

\subsection{The proof of  Theorem \ref{Thm:interiorofHT}}

Before proving Theorem \ref{Thm:interiorofHT}, we need the following lemma:
\smallskip

\begin{Lem}\label{Lem:ergodic}(Theorem 0.1 in \cite{BW})\,\,
Let $f$ be $C^2$, volume-preserving, partially hyperbolic and center
bunched. If $f$ is essentially accessible, then $f$ is ergodic, and in fact has
the Kolmogorov property.
\end{Lem}

\bigskip

{\bf Proof of Theorem \ref{Thm:interiorofHT}:}\,\,For $\forall f\in \mathcal{CPH}^1_\omega(M)$, by Lemma \ref{Lem:denseofstabacc},
there is a stably accessible perturbation $f_1\in \mathcal{CPH}^1_\omega(M)$. Applying Lemma \ref{Lem:denseofC2},
we can obtain anther perturbation $f_2$ which is $C^2$ and accessible.
Using Lemma \ref{Lem:ergodic}, $f_2$ is ergodic. Hence the subset consisting of all the ergodic diffeomorphisms is $C^1$ dense in $\mathcal{CPH}^1_\omega(M)$. Since it is well-known that this subset is a $G_\delta$ set, we denote its intersection with $int(\overline{HT})$ and the residual set determined in Lemma \ref{BV} by $\r$ as the residual set we need.

For any $f\in \mathbb{R}$, in the finest dominated splitting, there are
center bundle with dimension larger than 1 by Lemma
\ref{Lem:farfromtangency}. Then the center exponents along this
bundle equal, for, otherwise, suppose the exponents are different,
then by Lemma \ref{BV}, for almost every point, along the
orbit, there is a dominated splitting. Because $f$ is ergodic, this
dominated splitting can be extended on the whole manifold. This
contradicts our assumption that this splitting is a finest dominated
splitting.

\qed

\bigskip

{\bf{Acknowledgement.}}\,\,The authors thank Viana with whom we
talked the example in Section 3, and thank Wen and Gan for their
useful comments.

\bigskip

\end{document}